\documentclass[a4paper,12pt]{amsart}

\usepackage{fullpage}

\makeindex

\usepackage{xcolor}
\usepackage[
colorlinks=true,
citecolor=blue,
urlcolor=blue,
linkcolor=blue,
pdfborder={0 0 0},
breaklinks
]{hyperref}

\usepackage{pifont}
\usepackage{amsfonts,graphics,amsmath,amsthm,amscd,amssymb,amsmath,latexsym,euscript, enumerate,kotex,mathtools}
\usepackage{epsfig,url}
\usepackage{flafter}
\usepackage[all,cmtip,line]{xy}
\usepackage{array}
\usepackage[english]{babel}
\usepackage{overpic}
\usepackage{subfig}
\usepackage{multirow}
\usepackage{tikz-cd}
\usepackage[T1]{fontenc}
\usepackage{lmodern}
\usepackage{pifont}

\usepackage{microtype}
\usepackage{wrapfig}
\usepackage[shortlabels]{enumitem}
\setlist[enumerate, 1]{1\textsuperscript{o}}

\newtheorem{theorem}{Theorem}[section]
\newtheorem{lemma}[theorem]{Lemma}

\newtheorem{conjecture}[theorem]{Conjecture}

\theoremstyle{definition} 
\newtheorem{definition}[theorem]{Definition}
\newtheorem{definition-lemma}[theorem]{Definition-Lemma}

\theoremstyle{remark}
\newtheorem{remark}[theorem]{Remark}

\numberwithin{equation}{section}

\newcommand{\C}{\mathbb{C}}
\newcommand{\R}{\mathbb{R}}
\newcommand{\Z}{\mathbb{Z}}

\newcommand{\Q}{\mathbb{Q}}

\newcommand{\A}{\mathbb{A}}

\def\Proj{\operatorname{Proj}}
\def\Spec{\operatorname{Spec}}

\def\codim{\operatorname{codim}}

\newcommand{\floor}[1]{\left\lfloor #1 \right\rfloor}
\newcommand{\ceil}[1]{\left\lceil #1 \right \rceil}
\let\oldframe\frame
\renewcommand\frame[1][allowframebreaks]{\oldframe[#1]}

\title[On constructibility of klt type varieties and simple $D$-module]{On constructibility of klt type varieties and simple $D$-module}

\date{\today}

\begin{document}

\author[D.~Kim]{Donghyeon Kim}
\address[Donghyeon Kim]{Department of Mathematics, Ewha Womans University, 52 Ewhayeodae-gil, Seodaemun-gu, Seoul 03760, Republic of Korea}
\email{narimial0@ewha.ac.kr}

\begin{abstract}
In this paper, we prove that the combination of dense klt type and $\mathcal{D}$-simplicity property establishes the klt type property.
\end{abstract}

\maketitle
\allowdisplaybreaks

\section{Introduction}
The notion of \emph{of klt type} (for the definition, see Definition \ref{hyogok}) is a generalization of the notion of klt to non-$\Q$-Gorenstein varieties, and perhaps a more ``moral'' notion than klt. Since the condition of being klt is a constructible condition (cf. \cite[Lemma 4.44]{Kol23}), it is natural to ask whether the condition of being klt type is a constructible condition (rather than merely being klt, which requires the $\Q$-Gorenstein property in fibers).

\smallskip

\begin{conjecture}[{cf. \cite[Problem 4.7]{HJ26}}] \label{conjecture}
Let $X\to S$ be an affine family of normal varieties, and let $S'\subseteq S$ be a Zariski-dense subset of closed points such that $X_s$ is of klt type. Suppose $\eta\in S$ is the generic point. Then, $X_{\overline{\eta}}$ is of klt type. 
\end{conjecture}

Note that this conjecture can be viewed as a characteristic $0$ version of the conjecture posed by Schwede--Smith (cf. \cite[Question 7.2]{SS10}) and a local version of a conjecture posed by Choi--Li--Zhou (cf. \cite[Question 1.1]{CLZ25}). The major problem in proving Conjecture \ref{conjecture} (as well as \cite[Question 7.2]{SS10} and \cite[Question 1.1]{CLZ25}) is that we cannot simply lift a boundary $\Q$-Weil divisor on $X_s$ for $s\in S'$ to $X$.

\smallskip

Let us introduce another conjecture of Mallory.

\begin{conjecture}[{cf. \cite[Remark 8.5]{Mal21}}] \label{Mallory}
Let $X:=\Spec R$ be an lc type affine variety, and assume $\mathcal{O}_X$ is $\mathcal{D}_X$-simple. Then, $X$ is of klt type.
\end{conjecture}

Note that, as remarked in \cite[Remark 8.5]{Mal21}, if \cite[Problem 5.1.2]{HW02} and \cite[Question 7.2]{SS10} are true, then Conjecture \ref{Mallory} is also true. It is worth remarking that \cite[Proof of Theorem 5.27]{BJN19} proves that if $\bigoplus_{m\ge 0}\mathcal{O}_X(\floor{-mK_X})$ is a finitely generated $\mathcal{O}_X$-algebra and \cite[Problem 5.1.2]{HW02} is true, then Conjecture \ref{Mallory} follows.

\smallskip

Our main theorem is that if we merge the conditions of Conjecture \ref{conjecture} and \ref{Mallory}, then we have the klt type property.

\begin{theorem} \label{simple}
Let $X\to S$ be an affine family of normal varieties, and let $S'\subseteq S$ be a Zariski-dense subset of closed points such that $X_s$ is of klt type. Suppose $\eta\in S$ is the generic point, and assume $\mathcal{O}_{X_{\overline{\eta}}}$ is $\mathcal{D}_{X_{\overline{\eta}}}$-simple. Then, $X_{\overline{\eta}}$ is of klt type.
\end{theorem}

Usually, Conjecture \ref{conjecture} is considered a conjecture for birational geometry, and Theorem \ref{simple} states that the theory of differential operators can be involved in the problem of birational geometry.

\smallskip

Let us briefly explain how to deduce the theorem and how to overcome the boundary $\Q$-divisor lifting problem. By employing the argument in \cite[Corollary 11.2.22]{Laz04b}, we deduce that if $\mathrm{lct}^{\mathrm{dFH}}(X_s,\Delta_s,G_s)$ (for the definition, see Definition \ref{dFH log}) for an effective $\Q$-Weil divisor $0\le G\sim_{\Q} -K_X$, an affine analog of Tian's $\alpha$-invariant has a uniform lower bound, then $X_{\overline{\eta}}$ is of klt type (see also \cite[Proof of Proposition 4.4]{CP16}. Although \cite[Proof of Proposition 4.4]{CP16} is incorrect, its reliance on \cite[Proof of Theorem 11.2.21]{Laz04b} motivates our proof of Lemma \ref{yakmulsan}). The condition that ``$\mathcal{O}_{X_{\overline{\eta}}}$ is $\mathcal{D}_{X_{\overline{\eta}}}$-simple'' implies that there exists a differential operator $P$ that maps $G_s$ to $1$. Now, by employing Xu--Zhuang's argument in \cite{XZ26b}, we can bound $\mathrm{lct}^{\mathrm{dFH}}(X_s,\Delta_s,G_s)$ by $\frac{1}{\mathrm{ord}\,P}$ (cf. Lemma \ref{Kollar}), and we can prove the theorem. The crucial point is that the constant $r$ in Lemma \ref{Kollar} does not depend on $s$, and thus it can be a way to bypass the difficulty of lifting the boundary $\Q$-divisor.

\subsection*{Acknowledgement}
The author was partially supported by National Research Foundation of Korea (NRF) funded by the Ministry of Education (No. RS-2026-25555589) and the Ministry of Science and ICT (No. RS-2026-25506097).

\smallskip

ChatGPT is used during the development of this work, including initial discussions and possible proof strategies. In particular, the author used ChatGPT to digest the logic of \cite[Section 5]{XZ26b} and apply it to prove Lemma \ref{Kollar}. The author has independently verified the arguments and all mathematical claims in the manuscript and takes full responsibility for their correctness.

\section{Preliminaries}
Throughout this paper, unless otherwise stated. By a \emph{variety}, we mean an integral, separated scheme of finite type over $\C$. All varieties are assumed to be quasi-projective.

\smallskip

We now fix some notation and conventions that will be used throughout the paper. 

\begin{itemize}
    \item[(1)]
    A \emph{sub-couple} $(X,\Delta)$ consists of a normal variety $X$ and a $\Q$-Weil divisor $\Delta$ on $X$. If, in addition, $K_X+\Delta$ is $\Q$-Cartier, we call $(X,\Delta)$ a \emph{sub-pair}. For an effective $\Q$-Weil divisor $\Delta$, we say that $(X,\Delta)$ is a \emph{couple} and a \emph{pair}, respectively.

    \item[(2)]
    For a scheme $S$ and a point $s\in S$, we denote $\kappa(s)$ by the residue field of $\mathcal{O}_{S,s}$ and regard $s$ as the morphism $s:\Spec \kappa(s)\to S$.
    The associated geometric point is $\bar{s}:\Spec\overline{\kappa(s)}\to S$.
    For an integral scheme $X$, denote by $k(X)$ its function field.

    \item[(3)]
    If $\eta$ is the generic point of an integral scheme $S$, we write
    $\bar{\eta}:\Spec\overline{\kappa(\eta)}\to S$ and call
    $X_{\bar{\eta}}:=X\times_S \bar{\eta}$ the \emph{geometric generic fiber} of $X\to S$.

    \item[(4)]
    For a morphism $f:X\to S$ of schemes and a (geometric) point $s\to S$, set
    $$ X_s:=X\times_S s.$$
    If $\Delta$ is a divisor on $X$, we write $\Delta_s$ for its pullback to $X_s$ whenever it is defined.

    \item[(5)]
    For a morphism $\pi:X\to S$ of schemes and a point $s\to S$, and an ideal $\mathfrak{a}\subseteq \mathcal{O}_X$, we define
    $$ \mathfrak{a}_s:=\mathrm{Im}(\mathfrak{a}\otimes \kappa(s)\to \mathcal{O}_{X_s}).$$
    Moreover, if $f\in \mathcal{O}_X$, then we denote by $f_s$ to be the image of $f$ in $\mathcal{O}_{X_s}$.

    \item[(6)]
    For a morphism $f:X\to S$ between varieties, we say that $f$ is a \emph{fibration} if it is proper, surjective, and $f_*\mathcal{O}_X=\mathcal{O}_S$.
    
    \item[(7)] Let $(X,\Delta)$ be a sub-pair, $f:X'\to X$ a proper birational morphism between normal varieties, and $E$ a prime divisor on $X'$.
    The \emph{log discrepancy} of $E$ with respect to $(X,\Delta)$ is
    $$
      A_{X,\Delta}(E):= \mathrm{mult}_E\bigl(K_{X'}-f^*(K_X+\Delta)\bigr)+1.
    $$
    This definition does not depend on the choice of $f$. We say that $(X,\Delta)$ is \emph{sub-kawamata log terminal (sub-klt)} if $A_{X,\Delta}(E)>0$ for every prime divisor $E$ over $X$. If $(X,\Delta)$ is sub-klt and a pair, then we say that $(X,\Delta)$ is a \emph{klt}.

    \item[(8)]
    Let $\pi:X\to Z$ be a fibration of normal varieties and $D$ a $\Q$-Cartier divisor on $X$.
    The base ideal of the complete linear system $|D|$ is the image of the natural map
    $$
      \pi^*\pi_*\mathcal{O}_X(D)\otimes \mathcal{O}_X(-D)\longrightarrow \mathcal{O}_X,
    $$
    and is denoted by $\mathfrak{b}(X/Z,|D|)$.

    \item[(9)]
    For a variety $S$, by \emph{shrinking $S$}, we mean replacing $S$ by a variety $S'$ equipped with a quasi-finite morphism $S'\to S$
    (e.g., an open immersion or a finite cover), and replacing all data by their pullbacks to $S'$ while keeping the same notation.
\end{itemize}

For every variety $X$, let us define a left $\mathcal{O}_X$-algebra $\mathcal{D}_X$.

\begin{definition}[{cf. \cite[Definition 0G3Q]{Stacks}}]
Let $X$ be a variety. A \emph{differential operator} $P:\mathcal{O}_X\to \mathcal{O}_X$ of order $r$ is a $\C$-linear map such that for all local sections $g\in \mathcal{O}_X$, the map $s\mapsto D(gs)-gD(s)$ is a differential operator of order $r-1$. For the base case $r=0$, we define a differential operator of order $0$ to be an $\mathcal{O}_X$-linear map. We can define $\mathcal{D}_X$ as an $\mathcal{O}_X$-module such that for every affine $U\subseteq X$,
$$ \mathcal{D}_X(U):=\{P\mid P:\mathcal{O}_U\to \mathcal{O}_U\text{ is a differential operator}\}.$$
\end{definition}

We can define the notion of \emph{of klt type}.

\begin{definition} \label{hyogok}
Let $(X,\Delta)$ be a couple. We say that $(X,\Delta)$ is \emph{of klt type} if there exists an effective $\Q$-Weil divisor $\Delta'$ on $X$ such that $(X,\Delta+\Delta')$ is klt.
\end{definition}

\subsection{Valuation}
Throughout this subsection, we denote by $X$ a normal variety over an algebraically closed field $k$ of characteristic $0$.

\smallskip

A \emph{valuation} $\nu$ of $K(X)$ is a function $\nu:K(X)^{\times}\to \R$ satisfying the following conditions:
\begin{itemize}
    \item $\nu(fg)=\nu(f)+\nu(g)$;
    \item $\nu(f+g)\ge \min\{\nu(f),\nu(g)\}$;
    \item $\nu(a)=0$ for $a\in k$.
\end{itemize}
We also set $\nu(0)=+\infty$. Every valuation $\nu$ of $K(X)$ gives rise to a valuation ring $\mathcal{O}_{\nu}:=\{f\in K(X)\mid \nu(f)\ge 0\}$.

\smallskip

We say a valuation $\nu$ of $K(X)$ is \emph{centered on } $\xi=c_X(\nu)\in X$ if there exists an injection $\mathcal{O}_{X,\xi}\hookrightarrow \mathcal{O}_{\nu}$. We denote by $\mathrm{Val}_X$ the set of all valuations of $K(X)$ admitting a center on $X$, and we denote by $\mathrm{Val}_{X,\xi}$ the subset of $\mathrm{Val}_X$ consisting of valuations centered at $\xi$. Note that the center of a valuation is unique if it exists, due to the separatedness of $X$.

\smallskip

For an ideal $\mathfrak{a}\subseteq \mathcal{O}_X$, we define
$$ \nu(\mathfrak{a}):=\min\{\nu(f)\mid f\in \mathfrak{a}\},$$
and endow a topology on $\mathrm{Val}_X$ such that $\nu\mapsto \nu(\mathfrak{a})$ is continuous for every $\mathfrak{a}$. 

\smallskip

Let $(X,\Delta)$ be a klt pair. The \emph{log discrepancy} function
$$ A_{X,\Delta}:\mathrm{Val}_{X}\to \R\cup \{\infty\}, $$
is defined as in \cite{JM12} and \cite[Theorem 3.1]{BdFFU15}.

\subsection{Multiplier ideal sheaf}

In this subsection, we recall the definition of \emph{multiplier ideal sheaf}.

\smallskip

Let $(X,\Delta)$ be a pair, $\mathfrak{a}$ be an ideal in $\mathcal{O}_X$, and let $f:X'\to X$ be a log resolution of $(X,\Delta,\mathfrak{a})$. Let $F$ be an effective Cartier divisor on $X'$ such that $\mathfrak{a}\mathcal{O}_{X'}=\mathcal{O}_{X'}(-F)$. Let $\lambda\ge 0$. We define the \emph{multiplier ideal sheaf} $\mathcal{J}(X,\Delta,\mathfrak{a})$ as
$$ \mathcal{J}(X,\Delta,\mathfrak{a}^{\lambda})=f_*\mathcal{O}_{X'}(\ceil{K_{X'}-f^*(K_X+\Delta)-\lambda E}).$$
Note that if $\mathfrak{a}=\mathcal{O}_X(-D)$ for some effective Cartier divisor $D$ on $X$, then we write
$$ \mathcal{J}(X,\Delta+\lambda D):=\mathcal{J}(X,\Delta,\mathfrak{a}^{\lambda}).$$
We define the \emph{log canonical threshold} as
$$ \mathrm{lct}^{\mathfrak{q}}(X,\Delta,\mathfrak{a}):=\inf\{\lambda\mid \mathfrak{q}\nsubseteq \mathcal{J}(X,\Delta,\mathfrak{a}^{\lambda})\},$$
and
$$ \mathrm{lct}(X,\Delta,\mathfrak{a}):=\mathrm{lct}^{\mathcal{O}_X}(X,\Delta,\mathfrak{a}).$$
If $\mathfrak{a}=\mathcal{O}_X(-D)$ for some effective Cartier divisor, we set
$$ \mathrm{lct}(X,\Delta,rD):=\frac{1}{r}\mathrm{lct}(X,\Delta,\mathcal{O}_X(-D))$$
for a positive rational number $r$. Then, it is well-known
$$ \mathrm{lct}^{\mathfrak{q}}(X,\Delta,\mathfrak{a})=\inf_{\nu\in \mathrm{Val}^*_X}\frac{A_{X,\Delta}(\nu)+\nu(\mathfrak{q})}{\nu(\mathfrak{a})}$$
(cf. \cite[Corollary 6.9]{JM12}).

\smallskip

Let $X$ be smooth, $f:X\to S$ a surjective projective morphism, and let $L$ be an effective Cartier divisor on $X$. Then, we define that
$$ \mathcal{J}(X/S,\|L\|):=\bigcup_{0\le L'\sim_{\Q,S}L}\mathcal{J}(X,0,L').$$

\subsection{Kollár components}

The notion of \emph{the Kollár component} is first introduced in \cite{Pro00} and \cite{Xu14}, and we extend the definition. 

\begin{definition}
We define the notion of \emph{Kollár component} in two-steps:
\begin{itemize}
    \item[(1)] Let $x\in (X:=\Spec R,\Delta)$ be a klt singularity with a local ring $R$ essentially of finite type over $\C$. We say that a prime divisor $S$ over $x\in X$ is a \emph{Kollar component} if there exists a proper birational morphism $\pi:X'\to X$ such that
\begin{itemize}
    \item[(a)] $S$ is the unique exceptional divisor,
    \item[(b)] $-(K_{X'}+S+\pi^{-1}_*\Delta)$ is $\Q$-Cartier and ample over $X$, and 
    \item[(c)] $(X',S+\pi^{-1}_*\Delta)$ is plt.
\end{itemize}
The morphism $\pi:X'\to X$ is called the \emph{plt extraction} of the Kollár component $S$.
\item[(2)] Let $(X,\Delta)$ be a klt pair in general, and let $S$ be a prime divisor over $X$. We say that $S$ is a \emph{Kollár component} if
\begin{itemize}
    \item[(a)] $S\subseteq X$, or
    \item[(b)] there exists a (not necessarily closed) point $x\in X$ such that $S$ is a Kollár component of $(X_x,\Delta_x)$.
\end{itemize}
\end{itemize}
\end{definition}

Note that our definition is equivalent to the definition of the \emph{Kollár component} in \cite[Definition 2.3]{Zhu24a}.

\begin{lemma}[{cf. \cite[Lemma 2.4]{Zhu24a}}] \label{daljeonjeosuji}
Let $(X,\Delta)$ be a klt pair, and let $\mathfrak{a}\subsetneq \mathcal{O}_X$ be an ideal. Then, there exists a Kollár component computing $\mathrm{lct}(X,\Delta;\mathfrak{a})$.
\end{lemma}

\begin{proof}
Let $T$ be a prime divisor computing $\mathrm{lct}(X,\Delta;\mathfrak{a})$, and let $x\in X$ be the generic point of an irreducible component of the center of $T$.

\smallskip

We can substitute $\mathfrak{a}$ by $\mathfrak{a}_x+\mathfrak{m}^N_x$ with $N\cdot \mathrm{ord}_T(\mathfrak{m}_x)\ge \mathrm{ord}_T(\mathfrak{a})$. Indeed, since
$$ \mathrm{ord}_T(\mathfrak{a}_x+\mathfrak{m}^N_x)=\min\{\mathrm{ord}_T(\mathfrak{a}),N\cdot \mathrm{ord}_T(\mathfrak{m}_x)\}=\mathrm{ord}_T(\mathfrak{a}),$$
we obtain
$$ c:=\mathrm{lct}(X,\Delta;\mathfrak{a})\le \mathrm{lct}(X_x,\Delta_x;\mathfrak{a}_x+\mathfrak{m}^N_x)\le \frac{A_{X_x,\Delta_x}(T)}{\mathrm{ord}_T(\mathfrak{a}_x+\mathfrak{m}^N_x)}=\mathrm{lct}(X,\Delta;\mathfrak{a}).$$

\smallskip

Now, apply \cite[Lemma 2.7]{Xu20}. Then, there exists a Kollár component $S$ over $X$ such that
$$ A_{X_x,\Delta_x}(S)=c\cdot \mathrm{ord}_S(\mathfrak{a}_x+\mathfrak{m}^N_x)\le c\cdot \mathrm{ord}_S(\mathfrak{a}_x)\le A_{X_x,\Delta_x}(S),$$
and hence, $A_{X,\Delta}(S)=c\cdot \mathrm{ord}_S(\mathfrak{a})$.
\end{proof}

Note that we did not require $x\in X$ to be a localization of a variety at a closed point. We will prove a lemma that plays the crucial role in proving our main theorem, Theorem \ref{simple}.

\begin{lemma} \label{Kollar}
Let $x\in X:=\Spec R$ be an $n$-dimensional klt germ over $\C$ of characteristic $0$, and let $S$ be a Kollár component over $x\in X$. For every $f\in R$ and every order $r$ differential operator $P:R\to R$,
\begin{equation} \label{ineq}
\mathrm{ord}_S(Pf)\ge \mathrm{ord}_S(f)-rA_X(S).
\end{equation}
\end{lemma}

\begin{proof}
First of all, we treat the case of $\kappa(x)=\C$. For the case of $S\subseteq X$, we obtain $A_X(S)=1$ and $\mathrm{ord}_S(Pf)\ge \mathrm{ord}_S(f)-r$ by direct computation. Hence, we may assume that $S$ is not on $X$.

\smallskip

We will follow the argument in \cite[Theorem 5.10]{XZ26b}. The main idea is that we can translate the problem of proving (\ref{ineq}) into the problem of computing the maximal slope of a symmetric power of a tangent bundle. The difference is that \cite{XZ26b} uses $\Omega^{[j]}_{\mathcal{X}'}(\log \mathcal{S})|_{\mathcal{S}}$ because it is treating differential \emph{forms}, while we use $\mathrm{Sym}^{[j]}_{\mathcal{S}}T_{\mathcal{S}}$ because we are treating differential \emph{operators}.

\smallskip

\noindent \textbf{Step 1.} Let $\pi:(X',S)\to X$ be the plt extraction of $S$. By adjunction, there exists $\Delta_S$ on $S$ such that
$$ (K_{X'}+S)|_S=K_S+\Delta_S\,\text{ and }(S,\Delta_S)\text{ is log Fano.}$$
Set $H:=-(K_S+\Delta_S)$. Then, $H$ is ample. Since $K_{X'}+S=\pi^*K_X+A_X(S)S$, restriction to $S$ gives 
\begin{equation} \label{H}
H\sim_{\Q} -A_X(S)S|_S.
\end{equation}
Here, we have shrunk $X$ around $x$ so that a Cartier multiple of $K_X$ is linearly trivial.

\smallskip

Consider the index-one covering stack associated with $S$: $\rho:\mathcal{X}'\to X',\,\mathcal{S}:=\rho^*S$ (cf. \cite[Definition 5.3.1]{AH11}). The divisor $\mathcal{S}$ is Cartier. There is a smooth open substack $\mathcal{U}\subseteq \mathcal{X}'$ such that $S^{\circ}:=\mathcal{S}\cap \mathcal{U}$ is smooth and $\codim_{\mathcal{S}}(\mathcal{S}\setminus \mathcal{S}^{\circ})\ge 2$.

\smallskip

Let us explain why we use an index-one cover in the sense of a \emph{stack}, rather than merely an index-one covering that renders $\mathcal{S}$ Cartier. If we only work with an index-one cover as a \emph{scheme}, we can show that $X$ admits an open cover $\{U_i\}$ such that each $U_i$ has an index-one covering scheme $U'_i\to U_i$. However, to glue the $U'_i$ together globally, we need the framework of (Deligne–Mumford) stacks.

\smallskip

Define $\mathcal{I}:=\mathcal{O}_{\mathcal{U}}(-\mathcal{S}^{\circ})$ and $\mathcal{N}:=\mathcal{O}_{\mathcal{X}'}(\mathcal{S})|_{\mathcal{S}}$. Then, $\mathcal{N}$ is the normal line bundle on the index stack. Let $h:W\overset{\tau}{\to} \mathcal{S}\overset{\rho}{\to}S$ be a finite cover that factors through $\mathcal{S}$. Then, (\ref{H}) gives 
\begin{equation} \label{yanghak}
\begin{aligned}
\deg_H \mathcal{N}:=\frac{c_1(\tau^*\mathcal{N})\cdot (h^*H)^{n-2}}{\deg \tau}=-\frac{\deg \rho}{A_X(S)}H^{n-1}.
\end{aligned}
\end{equation}

\smallskip

\noindent \textbf{Step 2.} In this step, we move our differential operator $P$ to $\mathcal{S}^{\circ}$, and to $\mathcal{S}$.

\smallskip

By localization and étale pullback, $P$ defines a rational differential operator on $\mathcal{U}$, of order at most $r$, regular away from $\mathcal{S}^{\circ}$. On an étale chart $\Spec A$, choose coordinates adapted to $\mathcal{S}^{\circ}=(t=0)$ and $b\ge 0$ such that $t^bP$ is regular. Then, the expression of $P$ using the local coordinate on $A$ gives
$$ P(t^mA)\subseteq t^{m-b-r}A \text{ for }m\ge 0$$
(cf. \cite[Proposition 1.4.6]{Yek92}). Hence, we obtain that there exists an integer $q\ge 0$ such that \begin{equation}\label{P}
P(\mathcal{I}^m)\subseteq \mathcal{I}^{m-q}\text{ for }m \ge 0.
\end{equation}
Choose $q$ as a minimum. If $q=0$, the desired inequality follows immediately. Otherwise, $P$ induces maps
$$ D_m:\mathcal{I}^m/\mathcal{I}^{m+1}\to \mathcal{I}^{m-q}/\mathcal{I}^{m-q+1}, $$
one of which is nonzero by the minimality of $q$. Fix such an $m$. Using $\mathcal{I}^k/\mathcal{I}^{k+1}\cong \mathcal{N}^{-k}|_{\mathcal{S}^{\circ}}$, we thus obtain a nonzero differential operator $D_m:\mathcal{N}^{-m}|_{\mathcal{S}^{\circ}}\to \mathcal{N}^{q-m}|_{\mathcal{S}^{\circ}}$ of order $j\le r$.


\smallskip

Since $\mathcal{S}^{\circ}$ is smooth and $D_m$ has order $j$, its principal symbol gives a non-zero section
$$ 0\ne s\in H^0\left(\mathcal{S}^{\circ},\mathrm{Sym}^{j}T_{\mathcal{S}^{\circ}}\otimes \mathcal{N}^q|_{\mathcal{S}^{\circ}}\right)$$
(cf. \cite[(11.1.4)]{HTT08} with $\mathcal{V}=\mathcal{N}^{-m}|_{\mathcal{S}^{\circ}}\oplus \mathcal{N}^{q-m}|_{\mathcal{S}^{\circ}}$). Reflexifying this gives
\begin{equation}\label{nonzero}
0\ne s \in H^0\left(\mathcal{S},\mathrm{Sym}^{[j]}T_{\mathcal{S}}\otimes \mathcal{N}^q|_{\mathcal{S}}\right).
\end{equation}

\noindent \textbf{Step 3.} In this step, we compute the maximal slope of $\mathrm{Sym}^{[j]}T_{\mathcal{S}}\otimes \mathcal{N}^{-q}|_{\mathcal{S}}$ and reduce the inequality (\ref{ineq}) to the computation of the slope. Let $h:W\overset{\tau}{\to} \mathcal{S}\overset{\rho}{\to} S$ be a cover that factors through $\mathcal{S}$ as in Step 1.

\smallskip

On $\mathcal{S}^{\circ}$, put $\mathcal{C}^{\circ}:=\Omega^1_{\mathcal{U}}(\log \mathcal{S}^{\circ})|_{\mathcal{S}^{\circ}}$. Then, we obtain
$$ 0\to \Omega^1_{\mathcal{S}^{\circ}}\to \mathcal{C}^{\circ} \to \mathcal{O}_{\mathcal{S}^{\circ}}\to 0 $$
and
\begin{equation} \label{sm2}
0\to \mathcal{O}_{\mathcal{S}^{\circ}}\to (\mathcal{C}^{\circ})^{\vee}\to T_{\mathcal{S}^{\circ}}\to 0.
\end{equation}
Let $\imath:\mathcal{S}^{\circ}\hookrightarrow \mathcal{S}$ be the inclusion, and let $\mathcal{C}:=\imath_*\mathcal{C}^{\circ}$. We obtain that by \cite[Lemma 5.5, 5.6 and 5.9]{XZ26b},
\begin{equation} \label{XuZhuang}
\mu_{\mathrm{max},\rho^*H}(\mathcal{C})\le -\frac{\min\{1,\delta(S,\Delta_S)\}}{n}H^{n-1}<0.
\end{equation}
Here, $\delta(S,\Delta_S)$ is the $\delta$-invariant in K-stability (for the definition, see \cite{FO18} and \cite{BJ20}), and we define $\mu_{\mathrm{max},\rho^*H}$ and $\mu_{\mathrm{min},\rho^*H}$ for a stack as in \cite[Definition 5.7, Proof of Theorem 5.10]{XZ26b}. Dualizing gives $\mu_{\mathrm{min},\rho^*H}(\mathcal{C}^{\vee})>0$. Since $T_{\mathcal{S}}$ is a quotient of $\mathcal{C}^{\vee}$ in codimension $1$, (\ref{sm2}) implies 
\begin{equation}\label{remaining}
\mu_{\mathrm{min},\rho^*H}(T_{\mathcal{S}})>0
\end{equation}

\smallskip

Let us prove $\mathrm{deg}_{\rho^*H} T_{\mathcal{S}}=H^{n-1}$. Indeed, Since $K_{\mathcal{S}}\sim_{\Q}\rho^*(K_S+\Delta_S)$ (cf. \cite[Proof of Lemma 5.6]{XZ26b}), we obtain that $c_1(T_{\mathcal{S}})\equiv \rho^*H$. Hence, pulling it back gives $c_1((\tau^*T_{\mathcal{S}})^{**})\equiv h^*H$. Thus, by \cite[Definition 5.7]{XZ26b},
$$ 
\begin{aligned}
\deg_{h^*H}((\tau^*T_{\mathcal{S}})^{**}) & =c_1((\tau^*T_{\mathcal{S}})^{**})\cdot (h^*H)^{n-2}
\\ &=(h^*H)^{n-1}
\\ &=(\deg h)H^{n-1}.
\end{aligned}$$

\smallskip

Let the Harder-Narasimhan factors of $(\tau^*T_{\mathcal{S}})^{**}$ have rank $e_i$ and slopes $\lambda_1>\cdots>\lambda_k$. By $\mu_{\mathrm{min},\rho^*H}(T_{\mathcal{S}})>0$, every $\lambda_i$ is positive. Therefore,
$$ (\deg \rho)H^{n-1}=\sum^k_{i=1}e_i\frac{1}{\deg \tau}\lambda_i\ge \frac{1}{\deg \tau}\lambda_1=\frac{1}{\deg \tau}\mu_{\mathrm{max},h^*H}((\tau^*T_{\mathcal{S}})^{**})=\mu_{\mathrm{max},H}(T_{\mathcal{S}}),$$
The slope inequality (cf. \cite[Theorem 4.2]{GKP16}) with the definition of slope for stacks consequently yields 
$$\mu_{\mathrm{max},\rho^*H}(\mathrm{Sym}^{[j]}T_{\mathcal{S}})\le j\cdot \deg \rho\cdot H^{n-1}.$$

\smallskip

The non-zero section (\ref{nonzero}) gives an injection
$$ \mathcal{O}_{\mathcal{S}}\hookrightarrow \mathrm{Sym}^{[j]}T_{\mathcal{S}}\otimes \mathcal{N}^q|_{\mathcal{S}}.$$
Hence, the target has a maximal slope of at least zero. Therefore, we get
$$
\begin{aligned}
0& \le \mu_{\mathrm{max},\rho^*H}\left(\mathrm{Sym}^{[j]}T_{\mathcal{S}}\otimes \mathcal{N}^q|_{\mathcal{S}}\right)
\\ &=\mu_{\mathrm{max},\rho^*H}(\mathrm{Sym}^{[j]}T_{\mathcal{S}})+q\deg_H \mathcal{N}|_{\mathcal{S}} & (1)
\\ &\le \deg \rho\cdot \left(j-\frac{q}{A_X(S)}\right)H^{n-1},
\end{aligned}
$$
here we used (\ref{yanghak}) in (1). Since $H^{n-1},\deg \rho>0$, it follows that $q\le jA_X(S)\le rA_X(S)$. Finally, at the generic point of $\mathcal{S}$, the pullback of $f$ belongs to $\mathcal{I}^{\mathrm{ord}_S(f)}$. Thus, (\ref{P}) gives 
$$\mathrm{ord}_S(Pf)\ge \mathrm{ord}_S(f)-q\ge \mathrm{ord}_S(f)-rA_X(S).$$
We complete the proof for the case of $\kappa(x)=\C$.

\smallskip

\noindent \textbf{Step 4.} Let us prove the case that $\kappa(x)\ne \C$. Put $K:=\overline{\kappa(x)}$. The argument of Step 3 applies to each irreducible component of $\mathcal{S}_K$ over $K\cong \C$. Hence, by \cite[Theorem 1.3.7]{HL10}, we obtain
\begin{equation} \label{remaining2}
\mu_{\mathrm{min},\rho^*H}(T_{\mathcal{S}/\kappa(x)})>0.
\end{equation}

\smallskip

On $\mathcal{S}^{\circ}$, we have an exact sequence
\begin{equation} \label{remaining3}
0\to T_{\mathcal{S}^{\circ}/\kappa(x)}\to T_{\mathcal{S}^{\circ}/\C}\to \mathcal{O}_{S^{\circ}}\otimes_{\kappa(x)}\mathrm{Der}_{\C}(\kappa(x))\to 0 
\end{equation}
(cf. \cite[Lemma 04B2]{Stacks}). After reflexive extension, we obtain that
$$ 0\to T_{S/\kappa(x)}\to T_{\mathcal{S}/\C}\to \mathcal{D}\to 0 $$
with $\mathcal{D}^{**}=\mathcal{O}_{\mathcal{S}}\otimes_{\kappa(x)}\mathrm{Der}_{\C}(\kappa(x))$. Hence, whenever $\mathcal{D}\ne 0$, it is semistable of slope zero.  It follows from (\ref{remaining2}) that
$$ \mu_{\mathrm{min},\rho^*H}(T_{\mathcal{S}/\C})\ge 0.$$
Moreover, by (\ref{remaining3}), $c_1(T_{\mathcal{S}/\C})=c_1(T_{\mathcal{S}/\kappa(x)})\equiv \rho^*H.$
We can see $K_{\mathcal{S}/\C}\sim_{\Q}\rho^*(K_{S/\C}+\Delta_S)$ 
(cf. \cite[Proof of Lemma 5.6]{XZ26b}).

\smallskip

By the same argument as in Step 3,
$$ \deg_{\rho^*H}T_{\mathcal{S}/\C}=\deg_{\rho^*H}T_{\mathcal{S}/\kappa(x)}=H^{n-1}.$$
We can proceed with the argument in Step 3 to prove the desired inequality.
\end{proof}

\begin{remark}
By the same argument, we expect that we can improve the inequality in Lemma \ref{Kollar} to
$$ \mathrm{ord}_S(Pf)\ge \mathrm{ord}_S(f)-r\left(1-\frac{n-1}{n}\min\{1,\delta(S,\Delta_S)\}\right)A_X(S).$$
One can ask whether we can prove Lemma \ref{Kollar} for all valuations, not only for Kollár components. We do not know the proof, since it heavily relies on the geometric structure of a prime divisor (for example, adjunction and the use of \cite[Lemma 5.9]{XZ26b}).
\end{remark}

\subsection{Nakayama asymptotic order}
The purpose of this subsection is to introduce the notion of relative asymptotic vanishing order. We adopt this concept in order to prove Lemma \ref{yakmulsan} by following the strategy in \cite[Proof of Theorem 11.2.21]{Laz04b}.

\begin{definition}
Let $\pi:X\to Z$ be a projective morphism from a normal quasi-projective variety $X$ to a quasi-projective variety $Z$. Let $D$ be a $\pi$-big effective $\Q$-Cartier $\Q$-divisor on $X$. For a valuation $\nu\in \mathrm{Val}_X$, we define the \emph{asymptotic vanishing order}, denoted by $\sigma_{\nu}(X/Z,D)$, as follows: 
$$ \sigma_{\nu,m}(X/Z,D):=\frac{\nu(\mathfrak{b}(X/Z,|mD|))}{m}$$
for a sufficiently divisible $m$, and
$$ \sigma_{\nu}(X/Z,D):=\inf_{m\ge 1} \sigma_{\nu,m}(X/Z,D)=\inf_{m\ge 1}\frac{\nu(\mathfrak{b}(X/Z,|mD|))}{m}.$$
\end{definition}

\subsection{de Fernex-Hacon log discrepancy}
In this section, we define the notion of \emph{de Fernex-Hacon log discrepancy}, and prove some lemmas we need. For more details, see \cite{dFH09}. The reason we use the notion is that de Fernex-Hacon log discrepancy provides a good criterion for being of klt type (cf. \cite[Theorem 1.2]{dFH09}).

\begin{definition}
Let $X$ be a normal variety, and let $\nu\in \mathrm{Val}_X$ be a valuation over $X$. Suppose $D$ is a Weil divisor on $X$. We define
$$ \nu^{\natural}(D):=\nu^{\natural}(\mathcal{O}_X(-D)).$$
Moreover, we define
$$ \nu(D):=\lim_{m\to \infty}\frac{1}{m!}\nu^{\natural}(m!D).$$
Note that the limit exists by \cite[Lemma 2.8]{dFH09}. If $D$ is a $\Q$-Weil divisor, then we define
$$ \nu(D):=\frac{1}{m}\nu(mD)$$
for a sufficiently divisible $m$.
\end{definition}

\begin{definition}[{cf. \cite[Definition 7.1]{dFH09}}]
Let $(X,\Delta)$ be a couple, $f:X'\to X$ a proper birational morphism of normal varieties, and $E\subseteq X'$ a prime divisor. We define
$$ A_{X,\Delta,m}(E):=1+\mathrm{ord}_E(K_{X'})-\frac{1}{m}\mathrm{ord}^{\natural}_E(m(K_X+\Delta)),$$
and
$$ A_{X,\Delta}(E):=1+\mathrm{ord}_E(K_{X'})-\mathrm{ord}_E(K_X+\Delta)=\lim_{m\to \infty}A_{X,\Delta,m}(E).$$
For a general valuation $\nu\in \mathrm{Val}_X$, we can define $A_{X,\Delta,m}(\nu)$ and $A_{X,\Delta}(\nu)$ as in \cite[Section 3.2]{BdFFU15}.
\end{definition}

Let us define the notion of the \emph{de Fernex-Hacon log canonical threshold}.

\begin{definition} \label{dFH log}
Let $(X,\Delta)$ be a klt type couple, and let $\mathfrak{a}\subseteq \mathcal{O}_X$ be an ideal. We define
$$ \mathrm{lct}^{\mathrm{dFH}}(X,\Delta;\mathfrak{a}):=\inf_{\nu\in \mathrm{Val}^*_X}\frac{A_{X,\Delta}(\nu)}{\nu(\mathfrak{a})}.$$
If $\mathfrak{a}=\mathcal{O}_X(-G)$ for some effective Weil divisor $G$ on $X$, then we denote by
$$ \mathrm{lct}^{\mathrm{dFH}}(X,\Delta;G):=\mathrm{lct}^{\mathrm{dFH}}(X,\Delta;\mathcal{O}_X(-G)).$$
\end{definition}

\begin{remark}
Perhaps a more natural definition for $\mathrm{lct}^{\mathrm{dFH}}$ is that for every effective $\Q$-Weil divisor $0\le G\sim_{\Q} -(K_X+\Delta)$,
$$ \mathrm{lct}^{\mathrm{dFH}}(X,\Delta;G):=\inf_{\substack{\nu\in \mathrm{Val}^*_X \\  -\nu(-G)>0}}\frac{A_{X,\Delta}(\nu)}{-\nu(-G)},$$
because under the definition, $\mathrm{lct}^{\mathrm{dFH}}(X,\Delta;G)\ge 1$ if and only if $(X,\Delta+G)$ is lc. However, we do not know how to prove the next Lemma \ref{yakmulsan} with the definition. In particular, the third inequality of (\ref{missing}) is missing with the definition.
\end{remark}

We will prove a criterion for being klt type.

\begin{lemma} \label{yakmulsan}
Let $(X,\Delta)\to S$ be a family of couples, and let $S'\subseteq S$ be a Zariski-dense subset of closed points such that $(X_s,\Delta_s)$ is of klt type.

\smallskip

Suppose that there exists a positive integer $r$ such that $r(K_X+\Delta)$ is integral and a divisor $0\le rG\sim -r(K_X+\Delta)$ and a positive number $c>0$ such that $\mathrm{lct}^{\mathrm{dFH}}(X_s,\Delta_s,rG_s)>c$. Then, the geometric generic fiber of $(X,\Delta)\to S$ is of klt type.
\end{lemma}

\begin{proof}
We will use a family version of \cite[Corollary 11.2.22]{Laz04b}. We use a sub-klt pair on a log resolution rather than $(X,\Delta)$ itself to avoid an unwelcome Jacobian factor in subadditivity.

\smallskip

\noindent \textbf{Step 1.} In this step, we establish a bound of the log discrepancy $A_{X'_s,B_s}(\nu)$ with the Nakayama asymptotic order $\sigma_{\nu}$ to use the argument in \cite[Corollary 11.2.22]{Laz04b} (cf. (\ref{simple1})).

\smallskip

Let $\pi:X'\to X$ be a log resolution of $(X,\Delta,\mathcal{O}_X(-rG))$. Let $L:=\frac{1}{r}f^{\natural}(rG)$ be a $\Q$-Cartier divisor on $X'$ such that 
\begin{equation}\label{asdf}
\mathcal{O}_X(-rG)\cdot \mathcal{O}_{X'}=\mathcal{O}_{X'}(-rL),
\end{equation}
Choose $h\in K(X)^{\times}$ such that $r(G+K_X+\Delta)=\mathrm{div}_X(h)$. Define 
\begin{equation} \label{as1}
B:=-L-K_{X'/S}+\frac{1}{r}\mathrm{div}_{X'}(h).
\end{equation}

\smallskip

We may assume that
\begin{itemize}
    \item $\mathcal{O}_X(-rG)\otimes \kappa(s)=\mathcal{O}_{X_s}(-rG_s)$ for every $s\in S$, and
    \item $\pi:X'\to X$ is a fiberwise log resolution over $S$ (for the definition of \emph{fiberwise log resolution}, see \cite[Definition-Lemma 2.8]{Xu20}).
\end{itemize}
The first comes from \cite[Corollary 10.13]{Kol23}, and the second follows from \cite[Definition-Lemma 2.8]{Xu20}.

\smallskip

By the equality (\ref{asdf}), we have a surjection
$$ \pi^*\mathcal{O}_X(-rG)\to \mathcal{O}_{X'}(-rL).$$
Tensoring $\kappa(s)$ gives a surjection.
$$ \pi^*_s\mathcal{O}_{X_s}(-rG_s)\to \mathcal{O}_{X'}(-rL)\otimes \kappa(s).$$
Since we may assume that $rL\to S$ is flat by generic flatness (cf. \cite[Proposition 052A]{Stacks}),
$$ \mathrm{Tor}^{\mathcal{O}_S}_1(\mathcal{O}_{rL},\kappa(s))=0,$$
and thus considering the exact sequence
$$ 0\to \mathcal{O}_{X'}(-rL)\to \mathcal{O}_{X'}\to \mathcal{O}_{rL}\to 0,$$
we obtain an inclusion $\mathcal{O}_{X'}(-rL)\otimes \kappa(s)\subseteq \mathcal{O}_{X'_s}$, and the image is $\mathcal{O}_{X_s}(-rL_s)$. Hence, we obtain
\begin{equation} \label{듸듸<3}
\mathcal{O}_{X_s}(-rG_s)\cdot \mathcal{O}_{X'_s}=\mathcal{O}_{X'_s}(-rL_s).
\end{equation}

\smallskip

let $m=r\ell$ for a positive integer $\ell$, and let $u\in \mathcal{O}_{X_s}(mG_s)$ be a local section. Then, $u\cdot \mathcal{O}_{X_s}(-rG_s)^{\ell}\subseteq \mathcal{O}_{X_s}$, and thus by (\ref{듸듸<3}), $u\cdot \mathcal{O}_{X'_s}(-mL_s)\subseteq \mathcal{O}_{X'_s}$. Hence, $u\in \mathcal{O}_{X'_s}(mL_s)$. Since
$$ \mathcal{O}_{X_s}(mG_s)\to (\pi_s)_*\mathcal{O}_{X'_s}(mL_s)$$
is injective, we deduce that for every $m\in r\Z_{>0}$ and every $s\in S'$,
\begin{equation}\label{combining1} 
(\pi_s)_*\mathcal{O}_{X'_s}(mL_s)=\mathcal{O}_{X_s}(mG_s)=h^{-\frac{m}{r}}_s\mathcal{O}_{X_s}(-m(K_{X_s}+\Delta_s)).
\end{equation}
Note that the relative base ideal can be written as
\begin{equation} \label{combining2}
\begin{aligned}
\mathfrak{b}(X'_s/X_s,|mL_s|)&:=\mathrm{Im}(\pi^*_s\pi_{*s}\mathcal{O}_{X'_s}(mL_s)\otimes \mathcal{O}_{X'_s}(-mL_s)\to \mathcal{O}_{X'_s})
\\ &=\mathrm{Im}(\pi^*_s\mathcal{O}_{X_s}(mG_s)\otimes \mathcal{O}_{X'_s}(-mL_s)\to \mathcal{O}_{X'_s})
\\ &=\mathrm{Im}(\pi^*_s\mathcal{O}_{X_s}(mG_s)\to \mathcal{O}_{X'_s}(mL_s))\otimes \mathcal{O}_{X'_s}(-mL_s)
\\ &=(\mathcal{O}_{X_s}(mG_s)\cdot \mathcal{O}_{X'_s})\otimes \mathcal{O}_{X_s}(-mL_s).
\end{aligned}
\end{equation}
Hence, combining (\ref{combining1}) and (\ref{combining2}), we obtain
\begin{equation} \label{as2}
\sigma_{\nu,m}(X'_s/X_s,L_s)=\nu(L_s)-\frac{\nu(h_s)}{r}+\frac{1}{m}\nu(\mathcal{O}_X(-m(K_{X_s}+\Delta_s))).
\end{equation}
Thus,
\begin{equation} \label{jesan}
\begin{aligned}
A_{X_s,\Delta_s,m}(\nu)&=A_{X'_s}(\nu)+
\nu(K_{X'_s})-\frac{1}{m}\nu(\mathcal{O}_{X_s}(-m(K_{X_s}+\Delta_s))) &(1)
\\ &=A_{X'_s}(\nu)+\nu(K_{X'_s})+\nu(L_s)-\frac{\nu(h_s)}{r}-\sigma_{\nu,m}(X'_s/X_s,L_s)& (2)
\\ &=A_{X'_s,B_s}(\nu)-\sigma_{\nu,m}(X'_s/X_s,L_s), & (3)
\end{aligned}
\end{equation}
where
\begin{itemize}
    \item (1) comes from the definition of $A_{X_s,\Delta_s,m}(\nu)$,
    \item (2) follows from (\ref{as2}), and
    \item (3) follows from (\ref{as1}).
\end{itemize}
Moreover, by assumption, 
\begin{equation}\label{missing}
A_{X_s,\Delta_s}(\nu)\ge c\cdot \nu^{\natural}(rG_s)\overset{(1)}{=}cr\cdot \nu(L_s)\ge cr\cdot \sigma_{\nu}(X'_s/X_s,L_s),
\end{equation}
where (1) follows from (\ref{듸듸<3}). Hence, by (\ref{jesan}),
\begin{equation} \label{simple1}
A_{X'_s,B_s}(\nu)\ge (1+cr)\cdot \sigma_{\nu}(X'_s/X_s,L_s).
\end{equation}

\smallskip

\noindent \textbf{Step 2.} In this step, using the family version of \cite[Corollary 11.2.22]{Laz04b}, we can transform $\sigma_{\nu}$ in (\ref{missing}) with $\sigma_{\nu,m}$, and hence we can use \cite[Theorem 1.2]{dFH09}. As explained above, we work with sub-klt pair of a smooth variety in order to avoid an unwelcome Jacobian factor (cf. \cite{Tak13}).

\smallskip

Let $s\in S'$. Since each $X_s$ is of klt type, $(X'_s,B_s)$ is sub-klt (cf. (\ref{jesan})). Let us prove that there exists $C>0$ that does not depend on $s,m,\nu$ such that
\begin{equation} \label{simple2}
\sigma_{\nu,m}(X'_s/X_s,L_s)\le \sigma_{\nu}(X'_s/X_s,L_s)+\frac{C}{m}A_{X'_s,B_s}(\nu).
\end{equation}
We follow the proof of \cite[Theorem 11.2.21]{Laz04b}. Indeed, let $H$ be a very ample divisor on $X'$. Then, by relative Nadel vanishing theorem (cf. \cite[Generalization 11.2.15]{Laz04b}) and Castelnuovo-Mumford regularity (cf. \cite[Example 1.8.24]{Laz04a}),
$$ \mathcal{J}(X'_s/X_s,\|mL_s\|)\otimes \mathcal{O}_{X'_s}(K_{X'_s}+(\dim X/S+1)H+mL_s)$$
is $\pi_s$-generated. Hence, there exists an effective Cartier divisor $A$ on $X'$ such that
$$ \mathcal{J}(X'_s/X_s,\|mL_s\|)\otimes \mathcal{O}_{X'_s}(-A_s)\subseteq \mathfrak{b}(X'_s/X_s,|mL_s|)$$
(cf. \cite[Theorem 11.2.21]{Laz04b}). Moreover, subadditivity states that
$$ \mathfrak{b}(X'_s/X_s,|mqL_s|)\subseteq \mathcal{J}(X'_s/X_s,\|mqL_s\|)\subseteq \mathcal{J}(X'_s/X_s,\|mL_s\|)^q\text{ for }q\ge 1$$
(cf. \cite[Corollary 11.2.4]{Laz04b}). Hence,
$$ \sigma_{\nu,m}(X'_s/X_s,L_s)\le \sigma_{\nu}(X'_s/X_s,L_s)+\frac{\nu(A_s)}{m}.$$
Finally, we can fix $\varepsilon>0$ such that $(X'_s,B_s+\varepsilon A_s)$ is sub-klt. Hence, $\nu(A_s)\le \varepsilon^{-1}A_{X'_s,B_s}(\nu)$. Put $C:=\varepsilon^{-1}$.

\smallskip

\noindent \textbf{Step 3.} Hence, we obtain
$$ 
\begin{aligned}
A_{X_s,\Delta_s,m}(\nu)&=A_{X'_s,B_s}(\nu)-\sigma_{\nu,m}(X'_s/X_s,L_s) & (1)
\\ &\ge A_{X'_s,B_s}(\nu)-\sigma_{\nu}(X'_s/X_s,L_s)-\frac{C}{m}A_{X'_s,B_s}(\nu)& (2)
\\ &\ge A_{X'_s,B_s}(\nu)-\frac{1}{1+cr}A_{X'_s,B_s}(\nu)-\frac{C}{m}A_{X'_s,B_s}(\nu) & (3)
\\ &=\left(\frac{cr}{1+cr}-\frac{C}{m}\right)A_{X'_s,B_s}(\nu),
\end{aligned}
$$
here,
\begin{itemize}
    \item (1) follows from (\ref{jesan}),
    \item (2) comes from (\ref{simple2}), and
    \item (3) is due to (\ref{simple1}).
\end{itemize}
Thus, for every sufficiently divisible $m$, $A_{X_s,\Delta_s,m}(\nu)>0$. By \cite[Corollary 10.13]{Kol23} again, after shrinking $S$, for a fixed $m$, we may assume that
$$ \mathcal{O}_X(-m(K_X+\Delta))\otimes_{\mathcal{O}_S} \kappa(s)=\mathcal{O}_{X_s}(-m(K_{X_s}+\Delta_s)).$$
Thus, by the same argument as in the proof of (\ref{듸듸<3}),
$$ A_{X_{\overline{\eta}},\Delta_{\overline{\eta}},m}(E)>0\text{ for every prime divisor }E\text{ over }X_{\overline{\eta}}.$$
Hence, \cite[Theorem 1.2]{dFH09} gives the proof.
\end{proof}

\section{Proof of the main theorem}
In this section, we prove the main theorem of this paper, Theorem \ref{simple}.

\begin{proof}
After shrinking $S$, we may assume that $X:=\Spec R$ is affine. Let $0\le G\sim_{\Q,S} -K_X$ be any effective $\Q$-Weil divisor on $X$, $r$ be a positive integer such that $rG\sim_S -rK_X$, and let $0\ne f\in \mathcal{O}_X(-rG)\subseteq R$ be an element. After shrinking $S$ again, by generic flatness (cf. \cite[Proposition 052A]{Stacks}), we may assume that both $\mathrm{div}\,f$ and $rG$ are flat over $S$. Moreover, we obtain
$$ \mathrm{lct}^{\mathrm{dFH}}(X_s;(f_s))\le \mathrm{lct}^{\mathrm{dFH}}(X_s;rG_s)$$
for every $s\in S'$. We will prove that there exists a constant $c>0$ such that $\mathrm{lct}^{\mathrm{dFH}}(X_s;(f_s))>c$; if the proof is complete, then we have $X_{\overline{\eta}}$ is of klt type by Lemma \ref{yakmulsan}.

\smallskip

We denote the base change of $f$ to the geometric generic fiber by $f$ using an abuse of notation. Since $\mathcal{O}_{X_{\overline{\eta}}}$ is $\mathcal{D}_{X_{\overline{\eta}}}$-simple, we obtain that there exists a relative differential operator $P:\mathcal{O}_{X_{\overline{\eta}}}\to \mathcal{O}_{X_{\overline{\eta}}}$ over $S$ such that $P(f)=1$. After a generic finite base change of $S$, we obtain a differential operator $P_S:R\to R$ such that $P_S\otimes 1=P$. Moreover, $P_S(f)=1$; indeed, $P_S(f)\otimes 1=(P_S\otimes 1)(f\otimes 1)=P(f)=1$ by definition, and hence $P_S(f)=1$ since $R\to R\otimes \overline{\kappa(\eta)}, \,\,\,a\mapsto a\otimes 1$ is injective. We denote $P_s:=P_S\otimes \kappa(s)$ for $s\in S'$. Then, $P_s(f_s)=1$.

\smallskip

Let $s\in S'$. By \cite{BCHM} (cf. \cite[Theorem 92]{Kol08}), we obtain that $\bigoplus_{m\ge 0}\mathcal{O}_X(-mK_{X_s})$ is finitely generated. Set $X'(s):=\Proj_{\mathcal{O}_{X_s}}\bigoplus_{m\ge 0}\mathcal{O}_X(-mK_{X_s})$, and let $\pi(s):X'(s)\to X_s$ be the natural morphism. Then, we obtain that
$$ \mathcal{O}_X(-mK_{X_s})\cdot \mathcal{O}_{X'(s)}=\mathcal{O}_{X'(s)}(-mK_{X'(s)})\text{ for a sufficiently divisible }m,$$
and hence 
$$A_{X_s}(\nu')=A_{X'(s)}(\nu')\text{ for every }\nu'\in \mathrm{Val}_{X_s},$$
and
$$ \mathrm{lct}^{\mathrm{dFH}}(X_s;(f_s))=\mathrm{lct}(X'(s);(\pi(s)^*f_s)).$$
We can extend $P_s$ to $P_{X'(s)}:\mathcal{O}_{X'(s)}\to \mathcal{O}_{X'(s)}$. Now, by Lemma \ref{daljeonjeosuji}, we obtain that there exists a Kollár component $T$ over $X'(s)$ computing $\mathrm{lct}(X'(s);(\pi(s)^*f_s))$. Applying Lemma \ref{Kollar}, we obtain that
$$ 
\begin{aligned}
\mathrm{lct}^{\mathrm{dFH}}(X_s;(f_s))&=\mathrm{lct}(X'(s);(\pi(s)^*f_s))
\\ &=\frac{A_{X'(s)}(T)}{\mathrm{ord}_T(\pi(s)^*f_s)}
\\ &\ge \frac{1}{\mathrm{ord}\,P_{X'(s)}}\ge \frac{1}{\mathrm{ord}\,P}.
\end{aligned}$$
We complete the proof.
\end{proof}

\begin{remark}
One can ask whether there is a couple version of Theorem \ref{simple}. A problem is that we do not know the pair version of Lemma \ref{Kollar}. Indeed, the naive generalization of Theorem \ref{simple} 
$$\mathrm{ord}_S(Pf)\ge \mathrm{ord}_S(f)-rA_{X,\Delta}(S)$$
is false. 

\smallskip

Let $X:=\A^3_{\C}$, $D_i=(x_i=0)$ for $i=1,2,3$, and $\Delta:=\frac{5}{6}(D_1+D_2+D_3)$. Let $\pi:X'\to X$ be the blowup at the origin, and let $S$ be the exceptional divisor. Let $\Delta_{X'}$ be the strict transform of $\Delta$. Then,
$$ K_{X'}=\pi^*K_X+2S,\,\,\, \pi^*\Delta=\Delta_Y+\frac{5}{2}S,$$
and hence
$$ K_{X'}+S+\Delta_Y=\pi^*(K_X+\Delta)+\frac{1}{2}S, \,\,\, A_{X,\Delta}(S)=\frac{1}{2}.$$
Since $-S$ is $\pi$-ample, we obtain that $S$ is a Kollár component over $X$.

\smallskip

Let $P=\frac{\partial^r}{\partial^rx^r_1}$, and let $f=x^r_1$. Then, $Pf=r!$, $\mathrm{ord}_S(f)=r$, $\mathrm{ord}_S(Pf)=0$, and thus $\mathrm{ord}_S(Pf)< \mathrm{ord}_S(f)-rA_{X,\Delta}(S)$.
\end{remark}

\begin{remark}
As a reduction mod $p$ analog, one may wonder whether $D$-simplicity plus strongly $F$-regular type property imply the klt type property. One problem is that there is no valuative approach in $F$-singularity theory (cf. \cite[4.3. Positive characteristics]{XZ26a}).
\end{remark}

\bibliographystyle{halpha}
\bibliography{biblio}

\end{document}